\documentclass[10pt]{article}

\usepackage[margin=2cm]{geometry}

\usepackage[utf8x]{inputenc}
\usepackage[T2A]{fontenc}

\usepackage{intcalc, calc}
\usepackage[nomessages]{fp}
\usepackage{graphicx}
\usepackage{wrapfig}
\usepackage{float}

\usepackage{tikz}
\usepackage{color}
\usepackage{framed}
\usepackage{ragged2e}
\usepackage{amsmath, amssymb, amsfonts, amsthm}
\usepackage{wasysym}
\usepackage{soul, soulutf8}
\usepackage{calc}
\usepackage{mathtools,mathcomp}
\usepackage{pifont}
\usepackage{enumitem}

\usepackage{hyperref}

\usepackage{ifthen}

\def\divdots{\rlap{\raisebox{-1pt}{.}}{\rlap{\raisebox{2pt}{.}}\raisebox{5pt}{.}}}

\def\ndivby{\mathrel{
    \divdots
    \kern-0.35em\raise0.22ex\hbox{/}
}}

\renewcommand{\leq}{\leqslant}
\renewcommand{\geq}{\geqslant}

\def\fs{\kern 0.5em}
\newcounter{prcnt}
\newcounter{pucnt}
\newcommand{\prmain}[1]{ 
    \medskip%
    \setcounter{pucnt}{0}%
    \stepcounter{prcnt}%
    \noindent\textbf{%
    \theprcnt%
    \ifthenelse{\equal{#1}{1}}{*}{}%
    .}\fs%
}

\newcommand{\pumain}[1]{
    \stepcounter{pucnt}{%
    \noindent\bf(\alph{pucnt}%
    \ifthenelse{\equal{#1}{1}}{*}{}%
    )}\fs%
}

\newtheorem{theorem}{Theorem}

\newtheorem{lemma}{Lemma}

\newtheorem*{theorem*}{Теорема}
\newtheorem{corollary}{Corollary}
\theoremstyle{definition}

\newtheorem{example}{Example}
\newtheorem{conjecture}{Conjecture}

\title{The number of trees in distance-hereditary graphs and their friends}

\author{Danila Cherkashin$^{\mathrm{a}}$ and Pavel Prozorov$^\mathrm{b}$\\
{\small~a. Institute of Mathematics and Informatics, Bulgarian Academy of Sciences,}\\
{\small Sofia, Bulgaria}\\
{\small~b. St. Petersburg State University, Russia.}}

\begin{document}

\maketitle

\begin{abstract}
Counting the number of spanning trees in specific classes of graphs has attracted increasing attention in recent years. In this note, we present unified proofs and generalizations of several results obtained in the 2020s. The main method is to study the behavior of the vertex (degree) enumerator of a distance-hereditary graph under the operations of copying vertices.

Ehrenborg conjecture says that a Ferrer--Young graph maximizes the number of spanning trees among bipartite graphs with the same degree sequence. The second result of this paper is the equivalence of the Ehrenborg conjecture and its polynomial form.

\end{abstract}

\section{Introduction}

Let \( G = (V, E) \) be a finite, simple, connected, undirected graph, and let \( |V|=n \). For a vertex \( v \), let \( N_G(v) = \{ u \in V : vu \in E \} \) denote the neighborhood of vertex \( v \), and let \( \deg_G(v) \) denote the degree of vertex \( v \) in graph \( G \). For a subset of vertices \( U \subset V \), define the \textit{induced subgraph} \( G[U] \) as the graph whose vertices are elements of \( U \), and whose edges are those edges of \( G \) both of whose endpoints lie in \( U \). 

A \textit{tree} is a connected graph with no cycles. A disjoint union of trees is called
a \textit{forest}. A \textit{rooted tree} is a tree with one distinguished vertex (the root). A forest
that consists of rooted trees is called a \textit{rooted forest}. A \textit{spanning tree} in a connected
graph $G$ is a (not induced) subgraph of $G$ that is a tree containing all the vertices of $G$. A spanning
forest in a graph $G$ is a subgraph of G that is a forest containing all the vertices
of $G$.
Let \( \mathcal{S}(G) \) denote the set of all spanning trees of graph \( G \), and let \( \tau(G) = |\mathcal{S}(G)| \) be their number. 

We denote a complete graph on \( n \) vertices by \( K_n \), a complete bipartite graph by $K_{n,m}$, where 
$n,m$ are the sizes of the parts, a cycle graph on $n$ vertices by $C_n$ and a path graph on $n$ vertices by $P_n$.


Number the vertices of graph \( G \) as \( v_1, \dots, v_n \), assign variables \( x_1, \dots, x_n \) to them, and define the \textit{vertex spanning (degree) enumerator}
\[
P_G(x_1, x_2, \dots, x_n) = \sum_{T \in \mathcal{S}(G)} \prod_{v \in V} x_v^{\deg_T(v) - 1}.
\]

The \textit{extension} of a graph \( G \) is a graph \( \widetilde{G} \) obtained by adding a vertex \( v_0 \) to \( G \), connected to all vertices. It is well-known~\cite{pak1990enumeration} that \( P_{\widetilde{G}}(x_0, x_1, x_2, \dots, x_n) \) enumerates the \textit{rooted} spanning forests in graph \( G \), and that the identity
\[
P_G(x_1, x_2, \dots, x_n) \cdot (x_1 + x_2 + \dots + x_n) = P_{\widetilde{G}}(0, x_1, x_2, \dots, x_n)
\]
holds.

\section{Classes of graphs}

\subsection{Distance-hereditary graphs}

A \textit{distance-hereditary graph} is a connected graph in which any connected induced subgraph preserves the distances between any pair of vertices. More details on distance-hereditary graphs can be found, for example, in~\cite{bandelt1986distance} and the book~\cite{brandstadt1999graph}. In particular, this book provides the following equivalent definitions of distance-hereditary graphs:

\begin{itemize}
    \item[(i)] These are graphs in which any induced path is a shortest path.
    \item[(ii)] These are graphs in which any cycle of length at least five has two or more diagonals, and any cycle of exactly five has at least one pair of intersecting diagonals.
    \item[(iii)] These are graphs in which any cycle of length five or more has at least one pair of intersecting diagonals.
    \item[(iv)] These are graphs in which, for any four vertices \( u \), \( v \), \( w \), and \( x \), at least two of the three sums of distances \( d(u,v) + d(w,x) \), \( d(u,w) + d(v,x) \), and \( d(u,x) + d(v,w) \) are equal.
    \item[(v)] These are graphs that do not contain the following induced subgraphs: a cycle of length five or more, a gem, a house, or a domino (see Fig.~\ref{fig:forbidden}).
\begin{figure}[H]
    \centering
        \hfill  
   \begin{tikzpicture}[scale=1.8]
    
    \clip(1.5,0.8) rectangle (3.5,2.8);

    \coordinate(a1) at (2,1);
    \coordinate(a2) at (3,1);
    \coordinate(a3) at (3.3,1.9);
    \coordinate(a4) at (2.5,2.6);
    \coordinate(a5) at (1.7,1.9);

    \draw [blue,ultra thick] (a1)--(a2)--(a3)--(a4)--(a5);
    \draw [blue,ultra thick,dashed] (a5)--(a1);

    \fill (a1) circle (1.2pt);
    \fill (a2) circle (1.2pt);
    \fill (a3) circle (1.2pt);
    \fill (a4) circle (1.2pt);
    \fill (a5) circle (1.2pt);

\end{tikzpicture}
     \hfill
   \begin{tikzpicture}[scale=1.8,rotate=180]
    
    \clip(1.5,0.8) rectangle (3.5,2.8);

    \coordinate(a1) at (2,1);
    \coordinate(a2) at (3,1);
    \coordinate(a3) at (3.3,1.4);
    \coordinate(a4) at (2.5,2.6);
    \coordinate(a5) at (1.7,1.4);

    \draw [blue,ultra thick] (a1)--(a2)--(a3)--(a4)--(a5)--(a1);
    \draw [blue,ultra thick] (a1)--(a4)--(a2);
    
    \fill (a1) circle (1.2pt) node[right]{4};
    \fill (a2) circle (1.2pt) node[left]{3};
    \fill (a3) circle (1.2pt) node[left]{2};
    \fill (a4) circle (1.2pt) node[below]{1};
    \fill (a5) circle (1.2pt) node[right]{5};

\end{tikzpicture}
    \hfill 
    \begin{tikzpicture}[scale=1.8]
    
    \clip(1.5,0.8) rectangle (3.5,2.8);

    \coordinate(a1) at (2,1);
    \coordinate(a2) at (3,1);
    \coordinate(a3) at (3.3,1.9);
    \coordinate(a4) at (2.5,2.6);
    \coordinate(a5) at (1.7,1.9);

    \draw [blue,ultra thick] (a1)--(a2)--(a3)--(a4)--(a5)--(a1);
    \draw [blue,ultra thick] (a3)--(a5);

    \fill (a1) circle (1.2pt) node[below]{1};
    \fill (a2) circle (1.2pt) node[below]{5};
    \fill (a3) circle (1.2pt) node[below right]{4};
    \fill (a4) circle (1.2pt) node[below]{3};
    \fill (a5) circle (1.2pt) node[below left]{2};

\end{tikzpicture}
    \hfill 
    \begin{tikzpicture}[scale=1.8]
    
    \clip(1.5,0.8) rectangle (3.5,2.8);

    \coordinate(a1) at (1.7,1);
    \coordinate(a2) at (2.5,1);
    \coordinate(a3) at (3.3,1);
    \coordinate(a4) at (1.7,2.6);
    \coordinate(a5) at (2.5,2.6);
    \coordinate(a6) at (3.3,2.6);

    \draw [blue,ultra thick] (a1)--(a2)--(a3)--(a6)--(a5)--(a4)--(a1);
    \draw [blue,ultra thick] (a2)--(a5);

    \fill (a1) circle (1.2pt) node[below left]{6};
    \fill (a2) circle (1.2pt) node[below left]{1};
    \fill (a3) circle (1.2pt) node[below left]{2};
    \fill (a4) circle (1.2pt) node[below left]{5};
    \fill (a5) circle (1.2pt) node[below left]{4};
    \fill (a6) circle (1.2pt) node[below left]{3};

\end{tikzpicture}
    \caption{Forbidden induced subgraphs from left to right: long cycle, gem, house, and domino}
    \label{fig:forbidden}
\end{figure}
    \item[(vi)] These are graphs that can be constructed from a single vertex through a sequence of the following three operations:
    \begin{itemize}
        \item Adding a new pendant vertex connected by a single edge to an existing vertex.
        \item Replacing any vertex with a pair of vertices, each having the same neighbors as the removed vertex (a \textit{duplication without adding an edge}).
        \item Replacing any vertex with a pair of vertices, each having the same neighbors as the removed vertex, including the other vertex in the pair (a \textit{duplication with adding an edge}).
\end{itemize}
    
\end{itemize}

Another linear-algebraic definition was later introduced in~\cite{oum2005rank}.

Note that duplication operations (with or without adding an edge) can be applied not only to a distance-hereditary graph.

\subsection{Subclasses of distance-hereditary graphs}

\paragraph{A cograph} is a graph that does not contain a path on four vertices \( P_4 \) as an induced subgraph. 
A permutation graph is a graph in which vertices correspond to elements of a permutation, and edges correspond to transpositions. A permutation is called \textit{separable} if it does not contain the patterns 2413 and 3142. The class of cographs coincides with the class of graphs of separable permutations~\cite{bose1998pattern}. Since any graph in Fig.~\ref{fig:forbiddenThres} contains an induced \( P_4 \), any cograph is distance-hereditary.

\paragraph{A Ferrers--Young graph} is a bipartite graph whose vertices correspond to the rows and columns of a Ferrers--Young diagram, with an edge between a row and column if the corresponding cell is in the diagram. Clearly, any Ferrers--Young graph is distance-hereditary, as it preserves distances when transitioning to a connected induced subgraph (which is also a Ferrers--Young graph).

\paragraph{A threshold graph} is a graph that does not contain \( P_4 \), \( C_4 \), and \( 2K_2 \) as an induced subgraph. From this definition, it is clear that any threshold graph is a cograph and therefore a distance-hereditary graph.

\begin{figure}[H]
    \centering
    \begin{tikzpicture}[scale=1.5, every node/.style={circle, draw, fill=black, inner sep=1.5pt}, edge/.style={ultra thick, blue}]

\node (A1) at (0,1) {};
\node (A2) at (0,0) {};
\node (A3) at (1,1) {};
\node (A4) at (1,0) {};
\draw[edge] (A1) -- (A2);
\draw[edge] (A3) -- (A4);
\node [draw=none,fill=none] at (0.5,-0.3) {$2K_2$};

\node (B1) at (2.5,1) {};
\node (B2) at (2.5,0) {};
\node (B3) at (3.5,1) {};
\node (B4) at (3.5,0) {};
\draw[edge] (B1) -- (B2);
\draw[edge] (B3) -- (B4);
\draw[edge] (B1) -- (B3);
\node [draw=none,fill=none] at (3,-0.3) {$P_4$};

\node (C1) at (5,1) {};
\node (C2) at (5,0) {};
\node (C3) at (6,1) {};
\node (C4) at (6,0) {};
\draw[edge] (C1) -- (C2);
\draw[edge] (C2) -- (C4);
\draw[edge] (C4) -- (C3);
\draw[edge] (C3) -- (C1);
\node [draw=none,fill=none] at (5.5,-0.3) {$C_4$};

\end{tikzpicture}
    \caption{Forbidden induced subgraphs of a threshold graph}
    \label{fig:forbiddenThres}
\end{figure}

A special 2-threshold graph is a graph in which \( 2K_2 \), \( C_5 \), house, gem, net, Diamond+2P, \( W_4 + P \), and octahedron are forbidden.

\section{Methods}

\subsection{The matrix tree theorem}

The following theorem is one of the earliest results in algebraic combinatorics.

\begin{theorem}[Kirchhoff~\cite{kirchhoff1847ueber}, 1847]
    The number of spanning trees $\tau(G)$ in a graph $G$ is equal to any algebraic cofactor of the Laplace matrix $L(G)$.
\end{theorem}

In general, calculating determinants for specific graphs can be quite tedious.

\subsection{Rank-one perturbation}

Klee and Stamps noticed that it is possible to adjust the Laplace matrix slightly to simplify counting.

\begin{theorem}[Klee--Stamps~\cite{klee2020linear}, 2020]
    Let $a$ and $b$ be arbitrary vectors, and let $L(G)$ be the Laplace matrix of a graph $G$. Then
    \[
    \det \left( L(G) + a \cdot b^T \right) = \left( \sum_{i=1}^n a_i \right) \left( \sum_{i=1}^n b_i \right) \tau(G).
    \]
\end{theorem}

In a series of papers~\cite{klee2019linear,klee2020linear,go2022spanning}, this method was mainly used to bring the graph's Laplacian into an upper triangular form. The result is the following theorem.

\begin{theorem}[Go--Xuan--Luo--Stamps~\cite{go2022spanning}, 2022]
    The class of graphs for which there exist vectors $a$ and $b$ that transform $L(G) + a \cdot b^T$ into an upper triangular matrix coincides with the class of special 2-threshold graphs.
\end{theorem}

In this way, one can explicitly calculate the number of trees in a special 2-threshold graph. Moreover, it is possible to write out its vertex (degree) enumerator. We will discuss this in greater generality.

\subsection{Vertex (degree) enumerator}

Recall that the \textit{vertex (degree) spanning enumerator} is defined as the following polynomial:
\[
P_G(x_1, x_2, \dots, x_n) = \sum_{T \in \mathcal{S}(G)} \prod_{v \in V} x_v^{\deg_T(v) - 1},
\]
where the variables $x_v$ correspond to vertices, and $\deg_T(v)$ denotes the degree of vertex $v$ in tree $T$.

The following theorem is the main result of the paper~\cite{cherkashin2023stability}.

\begin{theorem}[Cherkashin--Petrov--Prozorov~\cite{cherkashin2023stability}] \label{th:oldmain}
    The statements (i)--(iii) are equivalent.
    \begin{itemize}
        \item[(i)] A graph $G$ is distance-hereditary.
        \item[(ii)] The polynomial $P_G$ factors into linear terms.
        \item[(iii)] The polynomial $P_G$ is real stable (i.e., it does not vanish when substituting any variables from the open upper complex half-plane).
    \end{itemize}
\end{theorem}

For counting the number of spanning trees and rooted forests (and specifically, in this paper), we are interested in a constructive proof of the implication from (i) to (ii). In the next chapter, we will prove it in a more general form than in~\cite{cherkashin2023stability}; see Corollary~\ref{cor:main}. However, we remind the reader of the definition of real stability for a polynomial (of several variables), which means that a polynomial with real coefficients does not become zero when substituting any values from the upper half-plane.

In the article~\cite{pozorov2025stability}, Theorem~\ref{th:oldmain} is generalized to weighted graphs with complex coefficients. The generalization shows the equivalence of the corresponding points (ii), (iii), and some weighted analogs of the definitions of distance-hereditary graphs. This allows us to define a class of weighted distance-hereditary graphs, with a somewhat loose naming convention, since the definition giving the class its name does not lend itself to such generalization.

\section{Main result}

The following theorem is the most general result on the factorization of the vertex spanning enumerator that we can prove.

\begin{theorem} \label{th:main}
Let graphs $G_1$ and $G_2$ be given, with a vertex $v_1$ marked in $G_1$ and a vertex $v_2$ marked in $G_2$. Consider a graph $H$ constructed as follows: take the disconnected union of $G_1$ and $G_2$, remove vertices $v_1$ and $v_2$, then add all edges between $N_{G_1}(v_1)$ and $N_{G_2}(v_2)$. Let the variables corresponding to vertices in $G_1$ be $x_1, \dots, x_n$, and those in $G_2$ be $y_1, \dots, y_m$, with the maximal numbers corresponding to the marked vertices. Then the following equality holds:
\[
P_H(x_1, \dots, x_{n-1}, y_1, \dots, y_{m-1})= P_{G_1} \left( x_1, \dots, x_{n-1}, \sum_{i \in N_{G_2}(v_2)} y_i \right) \cdot P_{G_2} \left( y_1, \dots, y_{m-1}, \sum_{i \in N_{G_1}(v_1)} x_i \right).
\]
\end{theorem}

\begin{proof}
Let us examine the spanning trees in $G_1$. Clearly, they are structured as follows: we take a certain forest $L_1$ in which the degree of vertex $v_1$ is zero ($\mathcal{S}_1$ is the set of such forests), and then attach each component of this forest to $v_1$ via a certain vertex. Let the forest $L_1$ have $t(L_1)$ connected components, then $N_{G_1}(v_1) = A_1 \cup A_2 \cup \dots \cup A_{t(L_1)}$, where all $A_i$ belong to the $i$-th connected component of the forest. Define the weight of a forest as $W(L_1) = \prod_{v \neq v_1} x_v^{\deg_{L_1}(v) - 1}$. Then we have
\[
P_{G_1}(x_1, x_2, \dots, x_{n-1}, x_n) = \sum_{L_1 \in \mathcal{S}_1} W(L_1) \prod_{i=1}^{t(L_1)} \left( \sum_{v \in A_i} x_v \right) x_n^{t(L_1) - 1}.
\]
Similarly for $G_2$ we can analogously define $\mathcal{S}_2$ and obtain
\[
P_{G_2}(y_1, y_2, \dots, y_{m-1}, y_m) = \sum_{L_2 \in \mathcal{S}_2} W(L_2) \prod_{i=1}^{t(L_2)} \left( \sum_{v \in B_i} y_v \right) y_m^{t(L_2) - 1}.
\]

Observe that a spanning tree in $H$ consists precisely of a pair of forests and a certain set of edges from $K_{N_{G_1}(v_1), N_{G_2}(v_2)}$. We now verify that the term in this polynomial corresponding to the pair of forests $(L_1, L_2)$ matches the product of the term for $L_1$ and the term for $L_2$ (that is, we verify this for fixed $L_1, L_2$).

Let $t_1 = t(L_1)$, $t_2 = t(L_2)$, and let $\sum_{v \in A_i} x_v = X_i$, $\sum_{v \in B_i} y_v = Y_i$. The term from the product is already calculated—it is
\[
\prod_{i=1}^{t_1} X_i \cdot \left( \sum_{v \in N_{G_2}(v_2)} y_v \right)^{t_1 - 1} \cdot \prod_{i=1}^{t_2} Y_i \cdot \left( \sum_{v \in N_{G_1}(v_1)} x_v \right)^{t_2 - 1}.
\]

Now, consider the term from $H$. Note that in $H$, it remains to select certain edges such that we obtain a spanning tree of the graph $K_{t_1, t_2}$, where vertices correspond to the sets $A_i$ and $B_j$. We call the corresponding spanning tree in the graph $K_{t_1, t_2}$ the type of tree, and let this spanning tree be $L_0$. Then let us look at the sum of all weights of trees of one type. Observe that each edge connecting vertices \( A_i \) and \( B_j \) in \( L_0 \) corresponds to any edge connecting a vertex from \( A_i \) and a vertex from \( B_j \), and the choice is independent for different edges. Therefore, the sum of the weights of all trees of type \( L_0 \) is
\[
\prod_{i=1}^{t_1}X_i^{\deg_{L_0}(A_i)} \cdot \prod_{i=1}^{t_2}Y_i^{\deg_{L_0}(B_i)}.
\]
Now it is clear that the total sum of all weights is
\[
\sum_{L_0 \in S(K_{t_1,t_2})}\prod_{i=1}^{t_1}X_i^{\deg_{L_0}(A_i)} \cdot \prod_{i=1}^{t_2}Y_i^{\deg_{L_0}(B_i)}= \prod_{i=1}^{t_1}X_i \cdot \prod_{i=1}^{t_2}Y_i \cdot P_{K_{t_1,t_2}}(X_1,X_2, \dots, X_{t_1}, Y_1, Y_2, \dots, Y_{t_2})=
\]
\[
= \prod_{i=1}^{t_1}X_i \cdot \prod_{i=1}^{t_2}Y_i \cdot \left( X_1+X_2 + \dots +X_{t_1}\right)^{t_2-1} \cdot (Y_1 +Y_2 + \dots Y_{t_2})^{t_1-1}=
\]
\[
= \prod_{i=1}^{t_1}X_i \cdot \prod_{i=1}^{t_2}Y_i \cdot \left( \sum_{v \in N_{G_1}(v_1)}x_v\right)^{t_2-1} \cdot  \left( \sum_{v \in N_{G_2}(v_2)}y_v\right)^{t_1-1},
\]
which is what we desired.
\end{proof}

\begin{corollary} \label{cor:main}
The vertex enumerator \( P_G \) of a distance-hereditary graph \( G \) factors into linear terms.
\end{corollary}

\begin{proof}
We use definition~(vi) of the class of distance-hereditary graphs.
Duplicating vertex \( v \) in a graph \( G \) with an edge corresponds to the substitution \( G_1 = G \), \( v_1 = v \), \( G_2 = K_3 \) (and any \( v_2 \)). Duplicating vertex \( v \) in a graph \( G \) without an edge corresponds to \( G_1 = G \), \( v_1 = v \), \( G_2 = (\{1,2,3\},\{\{1,2\},\{1,3\}\}) \), \( v_2 = 1 \).
In these cases, the second factor is obviously linear.
Adding a pendant vertex connected to \( v \) multiplies the enumerator by the variable \( x_v \) (if desired, this can be viewed as the substitution \( G_1 = G \), \( v_1 = v \), \( G_2 = (\{1,2,3\},\{\{1,2\},\{1,3\}\}) \), \( v_2 = 2 \)).
A simple induction completes the proof.
\end{proof}

The following application of Theorem~\ref{th:main} gives a non-linear factorization of the enumerator polynomial.

\begin{example}
Define the graph \( G \) as two disjoint copies of \( C_5 \) with all possible edges between them.
Then its vertex spanning enumerator factors into two irreducible polynomials of degree 4.
\end{example}

\section{Consequences and applications}

\subsection{Some classic results}

The results for the number of trees (and explicit formulas for the enumerators) of threshold graphs~\cite{merris1994degree}, Ferrers--Young graphs~\cite{ehrenborg2004enumerative}, as well as the classical formulas for complete graphs and complete multipartite graphs are obtained as direct corollaries.

\subsection{Some recent results on spanning trees enumeration}

\paragraph{Superprism.} In this paragraph \( n + i = i \).
The \textit{Superprism} \( S_n \) is a 4-regular graph on vertices \( v_i \), \( u_i \), \( i = 1, \dots, n \), whose edges are unions of complete bipartite graphs with parts \( \{u_i,v_i\} \) and \( \{u_{i+1},v_{i+1}\} \) for \( i = 1, \dots, n \).

It is clear that
\[
P_{C_n}  =  z_1 z_2 \dots  z_n \sum_{i=1}^n \frac{1}{z_iz_{i+1}}.
\]
Now we duplicate each vertex without adding an edge between copies, which replaces \( z_i \) by \( x_i + y_i \) for each \( i \) from 1 to \( n \).
By consecutively applying Theorem~\ref{th:main} (recall that duplication without an edge corresponds to \( G_1 = G \), \( v_1 = v \), \( G_2 = (\{1,2,3\},\{\{1,2\},\{1,3\}\}) \), \( v_2 = 1 \)), we obtain
\[
P_{S_n} = \prod_{i=1}^n (x_i+y_i)  \left ( \sum_{i=1}^n \frac{1}{(x_i+y_i)(x_{i+1}+y_{i+1})} \right)  \prod_{i=1}^n (x_i + y_i + x_{i+2} + y_{i+2}).
\]
Substituting \( x_i = y_i = 1 \) gives

\begin{corollary}[Bogdanowicz,~\cite{bogdanowicz2024number}, 2024]
For \( n \geq 4 \), we have
\[
\tau(S_n) = n \cdot 2^{3n-2}.
\]
\end{corollary}


\subsection{Some recent results on rooted spanning forests enumeration}

\begin{corollary}[Gao--Liu~\cite{gao2024tree}, 2024]
    The enumerator of rooted spanning forests of a cograph (the inversion graph corresponding to a separable permutation \( w \)) is
    \[
    P_{\widetilde{G_w}} (x_0, x_1, \dots , x_n) = \prod_{i=1}^{n-1} \left (x_0 + f^{(1)} + f^{(2)}_i \right),
    \]
    where
    \[
    f^{(1)}_i = \sum_{j \leq i, w(j)>w(i+1)} x_j \quad \text{and} \quad f^{(2)}_i = \sum_{j > i, w(j) < w(i)} x_j.
    \]
\end{corollary}

The corresponding formulas for complete multipartite graphs and split graphs have been studied in a recent series of works~\cite{dong2022counting},~\cite{yang2025counting}. The case of circulant graphs is considered in~\cite{mednykh2023cyclic}.

\section{Open questions}

\subsection{Ehrenborg’s conjecture}

\begin{conjecture}
Let \( G \) be a bipartite graph with parts \( V_1 \) and \( V_2 \). Then the inequality
\[
\tau(G) \cdot |V_1|\cdot |V_2| \leq \prod_{v\in V(G)} d_G(v) 
\]
holds.
\end{conjecture}

As we have already discussed, the equality holds for Ferrers--Young graphs. This conjecture can be strengthened, with equality still being achieved by substituting any non-negative weights on the vertices of a Ferrers--Young graph.

\begin{conjecture}
Let \( G \) be a bipartite graph with parts \( V_1 \) and \( V_2 \). Then, for any substitution \( x_1, \dots, x_n \geq 0 \), the inequality
\[
P_G(x_1, \dots, x_n) \cdot \left ( \sum_{v \in V_1} x_v \right ) \cdot \left ( \sum_{v \in V_2} x_v \right ) \leq  \prod_{v \in V(G)} \sum_{u \in N_G(v)} x_u 
\]
holds.
\end{conjecture}

The following theorem provides an additional numerical support for the Ehrenborg’s conjecture.

\begin{theorem}
    The polynomial form of Ehrenborg’s conjecture is equivalent to the usual form of Ehrenborg’s conjecture.
\end{theorem}

\begin{proof}
    If the polynomial form of Ehrenborg’s conjecture holds, then by setting each variable in \( P_G \) to 1, we obtain the usual form of Ehrenborg’s conjecture.

    To prove the reverse implication, suppose that the polynomial conjecture is false, meaning there exist a graph \( G \) and non-negative numbers \( x_1, \dots, x_n \) such that
    \begin{equation} \label{eq:wrongEre}
    P_G(x_1, \dots,x_n) \cdot \left (\sum_{i \in V_1} x_i \right) \left (\sum_{i \in V_2} x_i \right) > \prod_{i \in V} \left (\sum_{j \in N_G(i)}x_j \right).
    \end{equation}    
    By continuity, there exist positive rational numbers \( y_1, y_2, \dots, y_n \) such that inequality~\eqref{eq:wrongEre} still holds for them. We can multiply by a common denominator, obtaining that there exist natural numbers \( z_1, \dots, z_n \) for which inequality~\eqref{eq:wrongEre} holds.

    Then we construct a graph \( H \) where the usual form of Ehrenborg’s conjecture fails. To do this, replace each vertex \( v_i \) in \( G \) with \( z_i \) copies (i.e., perform \( z_i - 1 \) duplications without adding edges).
    Clearly, the resulting graph \( H \) is bipartite. By definition the number $\tau(H)$ of spanning trees in $H$ is equal to $P_H(1,1, \dots,1)$.
    Several applications of Theorem~\ref{th:main} with \( G_2 = (\{1,2,3\},\{\{1,2\},\{1,3\}\}) \), \( v_2 = 1 \) give
    \[
    P_H(1,1, \dots,1) = P_G(z_1, \dots,z_n) \cdot \prod_{i \in V(G)} \left(\sum_{j \in N_G(i)} z_j\right)^{z_i-1}.
    \]
    This holds because replacing vertices with weights \( k \) and 1 (being copies of each other) with a vertex of weight \( k+1 \) does not change the (weighted) degrees of vertices. Let us write Ehrenborg’s inequality for the graph \( H \):
    \[
    \tau(H) \cdot \left (\sum_{i \in V_1(G)} z_i \right) \left (\sum_{i \in V_2(G)} z_i \right ) \leq \prod_{i \in V(G)}\left(\sum_{j \in N_G(i)}z_i \right)^{z_i}.
    \]
    Dividing both sides by \( \prod_{i \in V}\left(\sum_{j \in N_G(i)} z_j\right)^{z_i-1} \) yields a contradiction with inequality~\eqref{eq:wrongEre}.
\end{proof}

\subsection{Factorization of the enumerator}

What can we say about a graph whose vertex enumerator factors? If all factors are linear, we know that the graph is distance-hereditary~\cite{cherkashin2023stability}. Recall that an example of a graph with a non-trivial factorization is two disjoint copies of \( C_5 \) with all possible edges between them.

Theorem~\ref{th:oldmain} immediately implies the recent result of Gao and Liu.

\begin{corollary}[Gao--Liu~\cite{gao2024tree}, 2024]
 The extended enumerator of a permutation graph \( G_w \) factors into linear terms if and only if \( G_w \) is a cograph.
\end{corollary}

In fact, we will prove a more general statement.

\begin{theorem} \label{th:forestcounting}
 The extended enumerator of a graph \( G \) factors into linear terms if and only if \( G \) is a cograph.
\end{theorem}

\begin{proof}
We need the following elementary lemma.

\begin{lemma}
The extension \( \widetilde{G} \) of a graph \( G \) is distance-hereditary if and only if \( G \) is a cograph.
\end{lemma}

\begin{proof}
If \( G \) is a cograph, it does not contain the forbidden graphs in Fig.~\ref{fig:forbidden}. Therefore, any potential forbidden subgraph must contain the added vertex \( v_0 \). Since \( v_0 \) is connected to all vertices, the only possible configuration is for \( v_0 \) to be vertex number 1 in the gem. However, the remaining vertices form \( P_4 \), which is also impossible.

If \( \widetilde{G} \) is distance-hereditary, \( G \) cannot contain \( P_4 \) since adding \( v_0 \) to \( P_4 \) forms an gem.
\end{proof}

Theorem~\ref{th:forestcounting} now follows from Theorem~\ref{th:oldmain}.

\end{proof}

\bibliographystyle{plain}
\bibliography{trees}

\end{document}